\documentclass[10pt,reqno,a4paper]{amsart}
\usepackage{graphicx} 

\usepackage{todonotes}
\usepackage{amsmath, stackrel}
\usepackage{amsfonts}
\usepackage{mathrsfs}  
\usepackage{amssymb}
\usepackage{amsthm}
\usepackage{hyperref}
\usepackage{graphicx}
\usepackage[noadjust]{cite}
\usepackage{caption}
\usepackage{bbm}
\usepackage{tikz}
\usetikzlibrary{decorations.pathreplacing}
\usepackage[T1]{fontenc}
\usepackage[utf8]{inputenc}
\usepackage{mathtools}

\usepackage[top=3.35cm, bottom=3.35cm, left=2.5cm, right=2.5cm, headsep=0.2in]{geometry}

\usepackage{fancyhdr}

\newtheorem{theorem}{Theorem}[section]
\newtheorem{remark}[theorem]{Remark}

\newtheorem{lemma}[theorem]{Lemma}

\newtheorem{definition}[theorem]{Definition}

\newcommand{\beq} {\begin{equation}}
\newcommand{\eeq} {\end{equation}}

\usepackage{verbatim}
\immediate\write18{texcount -tex -sum  \jobname.tex > \jobname.wordcount.tex}

\providecommand{\keywords}[1]
{
  \small	
  \textbf{\textit{Keywords---}} #1
}

\title{Stability of Haagerup Property under Graph Product}
\date{}
\author{Shubhabrata Das and Partha Sarathi Ghosh}
\keywords{Free product of groups, Graph product of groups, Haagerup property, Conditionally negative definite function\\ 2010 Mathematics Subject Classification: 20F65, 43A35, 20E06}

\begin{document}

\maketitle

\begin{abstract}
In this paper, we prove that any graph product of finitely many groups, all satisfying the Haagerup property (or Gromov's a-T-menability) also satisfies Haagerup property.

\end{abstract}

\section{Introduction}
Amenable groups were introduced by von Neumann, in connection with the Banach-Tarski paradox, as precisely the groups not admitting a paradoxical decomposition. Since then the notion of amenability was studied extensively and admits a number of equivalent characterizations. One such characterization is in terms of the existence of an approximate identity, consisting of finitely supported positive definite functions. A survey on this can be found in the book \cite{paterson2000amenability} and in the references therein.\\

The reduced group $C^*$-algebra, associated to a group $G$,  denoted by $C^*_\lambda(G)$, is the completion of the image of the left regular representation $\lambda:\mathbb{C}G\rightarrow \mathbb{B}(\ell^2 G)$, under the operator norm, where $\mathbb{C}G$ is the group ring of finitely supported $\mathbb{C}$-valued functions on $G$. A nuclear $C^*$-algebra means that the identity map (on itself) can be approximated by finite rank maps. $G$ is amenable if and only if $C^*_\lambda(G)$ is nuclear.\\

The most prominent of all non-amenable groups is the non-abelian free group $\mathbb{F}_n$, on $n$-generators. In \cite{haagerup1978example},  U. Haagerup proved that the non-nuclear reduced group $C^*$-algebra, $C^*_\lambda(\mathbb{F}_n)$, admits some approximation property, later termed as the Haagerup approximation property or simply \textit{Haagerup property}, by Akermann-Walter \cite{akemann1981unbounded} and Choda \cite{choda1983group} (see definition \ref{haagerup definition}). The class of groups with Haagerup property is quite large, containing all compact groups (finite groups in the discrete setting), all amenable groups (hence all abelian groups, solvable groups etc.), non-abelian free groups, fundamental groups of closed orientable surfaces etc. An equivalent definition for the groups with Haagerup property is the existence of a proper conditionally negative definite function on such groups. Groups acting properly on a tree \cite{haagerup1978example}, or on a $CAT(0)$-cube complex \cite{niblo1997groups}, admit such functions, and hence are shown to have Haagerup property. For a good survey on this topic see \cite{cherix2001groups}, by Cherix et al. \\

 Yet another way to characterise Haagerup property in a group, often a very useful one, is by showing such groups admit a proper affine isometric action on a Hilbert space. $SO(n,1),~ SU(n,1)$ admit this kind of action on a Hilbert space, and hence have Haagerup property. Bekka, Cherix and Valette showed that amenable groups also admit such actions, \cite{bekka1995proper}. In this light, Property (T) introduced by D. Kazhdan in \cite{kazhdan1967connection}, can be seen as a strong negation of Haagerup property, which is perhaps the reason behind M. Gromov calling Haagerup groups \textit{a-T-menable}, \cite{gromov1992asymptotic}. For example, as $SL_n(\mathbb{Z})$ has property (T) for $n\geq 3$, they do not have Haagerup property. Other such examples are $SL_n(\mathbb{Z})\ltimes \mathbb{Z}^n$, $Sp(n,1)$ for $n\geq 2$ etc. \\

Direct product of two groups with Haagerup property is again a group with Haagerup property. Also, Haagerup property is stable under taking free products and amalgamated products along finite subgroups \cite{cherix2001groups}. \textit{Graph product} of groups is a notion where groups corresponding to vertices in a finite graph are combined, and relators are commutators given by the edges (see definition \ref{graph product definition}). This notion in a way interpolates between the free product and direct product of groups. Since Haagerup property is stable under both, it seems a natural question to ask whether the Haagerup property remains preserved under any graph product, i.e. to ask: if we take groups $\{G_v\}_{v\in V(\Gamma)}$, one for each vertex in a finite graph $\Gamma$ and all with Haagerup property, does the graph product satisfy Haagerup property? We answer this question in the positive.

\begin{theorem}\label{main theorem}
Suppose $\Gamma$ is a finite graph and $\{G_v\}_{v\in V(\Gamma)}$ is a collection of Haagerup groups. Then the graph product $G(\Gamma)$ has Haagerup property.
\end{theorem}

This result was first proved by Y. Antolin and D. Dreesen in \cite{antolin2013haagerup}, following a very different approach. In some sense our proof is more direct. We show Haagerup property by producing a conditionally negative function on the graph product by gluing such functions on the vertex groups.\\

\section{Haagerup Property}\label{background}
Suppose $G$ is a group. From now on $G$ will always be a finitely generated group. Given a function $\phi:G\rightarrow \mathbb{C}$, it will induce a kernel $k_\phi:G\times G\rightarrow\mathbb{C}$ defined by $k_\phi(g,h):=\phi(h^{-1}g)$. By a positive definite function $\phi: G\rightarrow \mathbb{C}$, we mean that for each $n\in \mathbb{N}$, given any $\{c_1,c_2,\ldots,c_n\}\subset \mathbb{C}$ and $\{g_1,g_2,\ldots,g_n\}\subset G$ we have: $$\sum_{i,j=1}^nc_i\overline{c_j}k_\phi(g_i,g_j)=\sum_{i,j=1}^nc_i\overline{c_j}\phi(g_j^{-1}g_i)\geq 0.$$ 

The set $\ell^2G$ is the space of all square summable functions on $G$, which will form a Hilbert space with respect to the inner product $\langle f_1,f_2\rangle:=\sum_{g\in G} f_1(g)\overline{f_2(g)}$. $(\ell^2G)_1$ is the unit sphere of the Hilbert space. An alternative characterization of $\phi$ being positive definite is (\cite{de1989propriete}): the existence of an orthogonal representation $\pi:G\rightarrow \mathbb{B}(\ell^2G)$ and $\xi \in (\ell^2G)_1$ such that, $$\phi(g)=\langle\pi(g)\xi,\xi\rangle$$

Similarly a $\phi$ will be called conditionally negative definite (CND) if for any $n\in \mathbb{N}$ the $n\times n$ matrix 
$\begin{bmatrix}
K_{ij}=\phi(g_j^{-1}g_i)
\end{bmatrix}$ is a conditionally negative definite matrix for arbitrary choice of subsets $\{g_1,g_2,\ldots,g_n\}\subset G$. An equivalent characterization will be the existence of a Hilbert space $\mathcal{H}$ and a map $R:G\rightarrow \mathcal{H}$ such that, 
\begin{equation}\label{cnd condition}
  k_\phi(g,h)=\phi(h^{-1}g)=||R(g)-R(h)||^2  
\end{equation}

Given any kernel $k$ on $G$, either positive definite or conditionally negative definite, it will induce a function $\phi_k$ of same type on $G$ defined by the relation $\phi_k(g)=k(g,e)$, if $k$ is a $G$-invariant i.e. $k(g,h)=k(h^{-1}g,e)$ for all $(g,h)\in G\times G$. The interaction of positive definite kernels and conditionally negative definite is due to Schoenberg.

\begin{lemma}\label{schoenberg}(\cite{schoenberg1938metric}, Theorem D.11\cite{brown2008textrm})\hspace{1em}
$k:G\times G\rightarrow\mathbb{C}$ is conditionally negative definite if and only if $e^{-k}:G\times G\rightarrow\mathbb{C}$ is positive definite.
\end{lemma}

\begin{definition}\label{haagerup definition}
A group $G$ is called Haagerup group (or a group with Haagerup property) if there is a sequence of functions $\phi_n:G\rightarrow \mathbb{C}$ with the following properties:
\begin{itemize}
    \item[a)] $\phi_n(g)\rightarrow 1$ point-wise
    \item[b)] each $\phi_n$'s vanishes at $\infty$ i.e. $\phi_n(g)\rightarrow 0$ as $|g|\rightarrow \infty$,
    \item[c)] each $\phi_n$'s are positive definite.
\end{itemize}
\end{definition}
Suppose one has a proper function $\phi$ on $G$, which is conditionally negative definite, then form lemma \ref{schoenberg}, the sequence $\phi_n=e^{-\frac{\phi}{n}}$ will satisfy all the conditions of definition \ref{haagerup definition}. The other direction i.e. given a sequence $\phi_n$ with above properties one can construct a proper conditionally negative definite function $\phi$ on $G$, see \cite{cherix2001groups} for the proof of this fact. Therefore we have an equivalent definition:
\begin{definition}\label{hageerup definition 2}
    $G$ has Haagerup property if there is a proper, symmetric, conditionally negative definite function on $G$.
\end{definition}

Whenever $G$ is Haagerup group, we have a proper CND function $\phi$, and a map $R$ from $G$ to a Hilbert space $\mathcal{H}$, satisfying equation (\ref{cnd condition}). We will call this data, denoted by $(G,\phi,\mathcal{H},R)$ a Haagerup tuple corresponding to $G$.
\section{Graph Product}\label{graph product}

Suppose $\Gamma$ is a finite graph, without loops and bigons. $V(\Gamma)$ is vertex set of $\Gamma$ and $E(\Gamma)$ is the set of edges of $\Gamma$. An unordered edge between vertices $v$ and $w$, is denoted by $[v,w]$. Also let $\{G_v\}_{v\in V(\Gamma)}$  be a fixed family of groups. A graph product $G(\Gamma)$ of $\{G_v\}_{v\in V(\Gamma)}$ is obtained by combining $G_v$'s along $\Gamma$.   
\begin{definition}\label{graph product definition}
     For a finite graph $\Gamma$ and a family of groups $\{G_v\}_{v\in V(\Gamma)}$, the graph product $G(\Gamma)$ is the quotient $*_{v\in V(\Gamma)}G_v\Big/\langle\langle \{[G_v,G_w]:[v,w]\in E(\Gamma)\}\rangle\rangle$.

\end{definition}

\begin{figure}[ht!]
\centering
\parbox{5cm}{
\begin{center}
\begin{tikzpicture}[scale=0.33] 
\draw [fill] (1,0) circle [radius=5pt] node[below]{$G_1$};
\draw [fill] (5,0) circle [radius=5pt] node[below]{$G_2$};
\draw [fill] (6,3.5) circle [radius=5pt] node[right]{$G_3$};
\draw [fill] (0,3.5) circle [radius=5pt] node[left]{$G_4$};
\draw [fill] (3,6) circle [radius=5pt] node[above]{$G_5$};
\end{tikzpicture}
\end{center}
\caption{$\Gamma$ is disconnected\\ \centering{$G(\Gamma)=*_{i=1}^5G_i$}}
\label{fig:2figsA}}
\qquad
\begin{minipage}{5cm}
\begin{center}
\begin{tikzpicture}[scale=0.33] 
\draw [fill] (1,0) circle [radius=5pt] node[below]{$G_1$};
\draw [fill] (5,0) circle [radius=5pt] node[below]{$G_2$};
\draw [fill] (6,3.5) circle [radius=5pt] node[right]{$G_3$};
\draw [fill] (0,3.5) circle [radius=5pt] node[left]{$G_4$};
\draw [fill] (3,6) circle [radius=5pt] node[above]{$G_5$};
\draw (1,0) -- (5,0) -- (6,3.5) -- (0,3.5) -- (1,0) -- (3,6) -- (5,0) -- (0,3.5) -- (3,6) -- (6,3.5) -- (1,0);
\end{tikzpicture}
\end{center}
\caption{$\Gamma$ is complete\\\centering{$G(\Gamma)=\oplus_{i=1}^5G_i$}}
\label{fig:2figsB}
\end{minipage}
\end{figure}

The free product and the direct product of groups are the two extreme examples of graph product of groups. One can easily see that if $\Gamma$ is a graph with $n$-vertices and no edges then $G(\Gamma)$ is the free product of the vertex groups, and if $\Gamma$ is a complete graph with $n$-vertices, then $G(\Gamma)$ is the direct product (see the above figures). Graph product of groups were studied by Green in her thesis \cite{green1990graph}.\\

For $\{G_v\}_{v\in V(\Gamma)}$, clearly the graph product $G(A)\leq G(\Gamma)$, where $A$ is a subgraph of $\Gamma$. Define $lk(v):=\{w\in V(\Gamma):[v,w]\in E(\Gamma)\}$ and $st(v):=lk(v)\cup \{v\}$, for $v\in V(\Gamma)$. By $ST(v)$ we denote the maximal subgraph of $\Gamma$, whose vertex set is $st(v)$. Then $G(st(v))$ simply means the graph product $G(ST(v))$. Each element $g\in G(\Gamma)$ can be written in the form $$g=g_1g_2\ldots g_m$$ where each $g_i\in G_{v_i}$, for $v_i\in V(\Gamma)$, for some $m$. If $[v_i,v_{i+1}]\in E(\Gamma)$ then $$g=g_1g_2\ldots g_m=g_1\ldots g_{i-1}g_{i+1}g_{i}g_{i+2}\ldots g_m$$
We call the above operation a `shuffle'. One can perform the following three types of operations on the form $g=g_1g_2\ldots g_m$:
\begin{itemize}
    \item[a)] perform shuffle, 
    \item[b)] drop identity element if it appears,  
    \item[c)] multiply $g_i$ and $g_{i+1}$ if both come from $G_{v_i}$
\end{itemize}
\begin{remark}
Shuffle does not alter the number of elements. Operations (b) and (c) reduce the number of elements by $1$.
\end{remark}

A form $g=g_1g_2\ldots g_m$, is called \textit{reduced} if the number of elements remains the same after performing the above operations any number of times. This reduced form is the usual normal form in case of free product, which is uniquely determined. The presence of $\mathcal{R}_\Gamma$ relations in a general graph product $G(\Gamma)$ prevents the reduced form to be unique. However, the number of elements in any reduced form of a $g\in G(\Gamma)$, is unique \cite[Corollary 3.13]{green1990graph}. Therefore it will induce a length function $l_r$, given by the reduced length $|\cdot|_r$ on $G(\Gamma)$ i.e. 
\begin{equation}\label{definition l}
    \begin{array}{rl}
         l_r:G(\Gamma) & \rightarrow \mathbb{R} \\
        g & \mapsto l_r(g)=|g|_r:=m 
    \end{array}
\end{equation}
where the reduced length $|g|_r$ of a $g\in G(\Gamma)$ in reduced form $g=g_1g_2\ldots g_m$, is $m$.
\begin{remark}
    Whenever $g_1g_2\ldots g_m$ is a reduced form, and $g_i,g_j\in G_v$, then there is some $q\in \mathbb{N}$ and $u\notin st(v)$ such that $i<q<j$, and $g_q\in G_u$.
\end{remark}

\section{Theorem: Constructing CND functions on $G(
\Gamma)$}\label{theorem}
Let $\Gamma$ be a finite graph, and $\{G_v\}_{v\in V(\Gamma)}$ be a $\Gamma$-family of groups with Haagerup property. From definition, groups $G_v$'s with Haagerup property come equipped with proper conditionally negative definite functions $\phi_v:G_v\rightarrow \mathbb{R}$, Hilbert spaces $\mathcal{H}_v$'s with maps $R_v:G_v\rightarrow\mathcal{H}_v$ such that:
$$\phi_v(h^{-1}g)=||R_v(g)-R_v(h)||^2~~~\text{for all}~g,h\in G_v.$$
Our aim is to construct a proper conditionally negative definite function on $G(\Gamma)$, denoted $\phi_{\Gamma}$, combining the $\phi_v$'s on $G_v$'s. To that end, we will use a CAT(0)-cube complex structure that comes along with the graph product $G(\Gamma)$. Existence of a certain `wall structure' on $G(\Gamma)$ gives rise to a finite dimensional $CAT(0)$ cube complex $X$, with a $G(\Gamma)$ action on it. This is due to Reckwerdt \cite{reckwerdt2017weak} and Chatterji-Niblo \cite{chatterji2005wall}. The distance function on $G(\Gamma)$ defined by 
\begin{equation*}
    \begin{array}{rl}
         d_r:& G(\Gamma)\times G(\Gamma)\rightarrow \mathbb{R}\\
         & ~~~~~(g,h)\longmapsto |h^{-1}g|_r
    \end{array}
\end{equation*}
is `equivalent' to $d_X$, the metric on $X$. Notice that $l_r(g)=d_r(g,e)$. More specifically, in \cite{reckwerdt2017weak} Reckwerdt showed that, for any $g,h\in G(\Gamma)$ and for some $x_0\in X$, one has
\begin{equation}\label{comparison metric}
    d_X(gx_0,hx_0)=2d_r(g,h).  
\end{equation}
It follows from the above equality that the metric $d_r$ on $G(\Gamma)$ is conditionally negative definite since $d_X$ is such on the vertex set of $X$ \cite{niblo1997groups}. 

In order to obtain a proper CND function on $G(\Gamma)$, we will combine the metric $d_r$ with the CND functions $\{\phi_v\}$ along the CAT(0) cube complex $X$. First we give an outline in the free product case, and the general case follows similarly. 

\subsection*{Haagerup property for free product}\label{h property for free product}
Let us consider $G$ be the free product of $A$ and $B$, two groups with Haagerup property, and let $(A,\phi_A,\mathcal{H}_A,R_A)$, $(B,\phi_B,\mathcal{H}_B,R_B)$ be two associated Haagerup tuples. Let $T$ be the Bass-Serre tree for $G$. Consider the Hilbert space $\mathcal{H}=\oplus_{t\in T}\mathcal{H}_t$, where $\mathcal{H}_t$ is either $\mathcal{H}_A$ or $\mathcal{H}_B$, depending on whether $t$ belongs to $G/A$ or $G/B$, respectively. Also assume $R_A(e_A)=0_{\mathcal{H}_A}$ and $R_B(e_B)=0_{\mathcal{H}_B}$.\\

Any non-trivial element $g$ in $A*B$ has a unique normal form $g=g_1g_2\cdots g_n$, with $g_i\in A~\text{or}~B$. We use this normal form to define $R:G\rightarrow \mathcal{H}$ as follows:
\begin{equation}\label{definition of R}
         R(g):=
            \begin{cases}
                R_1(g_1)_{G(1)}\oplus R_2(g_2)_{g_1G(2)}\oplus \cdots\oplus R_n(g_n)_{g_1g_2\cdots g_{n-1}G(n)}, ~~\text{if}~g\neq e\\
                0_{\mathcal{H}}~~~~~~~~~~~~~~~~~~~~~~~~~~~~~~~~~~~~~~~~~~~~~~~~~~~~~~~~~~~~~~~~~~~\text{if}~g=e\\
            \end{cases} 
\end{equation}
where $g_i\in G(i)$, which is either $A$ or $B$ and $R_i$ is either $R_A$ or $R_B$ according to the position.\\
\begin{lemma}\label{group inv kernel in free product}
    The kernel $k$ on $G\times G$, defined by the equation, $$k(g,h):=||R(g)-R(h)||^2$$  is $G$-invariant i.e. for any $f,g,h\in G$ we have $k(fg,fh)=k(g,h)$. 
\end{lemma}
\begin{proof}
It is sufficient to proof that for any $g,h\in G$ we have $k(g,h)=k(h^{-1}g,e)$. Suppose $g=g_1g_2\cdots g_n$ and $h=h_1h_2\cdots h_m$ are their normal form with $g_i=h_i$ for $1\leq i\leq p\leq \min\{n,m\}$. Then the normal form of $h^{-1}g$ is $h_m^{-1}h_{m-1}^{-1}\cdots h_{p+2}^{-1}cg_{p+2}\cdots g_{n-1}g_n $ where $c=h_{p+1}^{-1}g_{p+1}$. We omit the positions of the vectors in suffix in the definition of $R$ to reduce clutter. Now observe, the components of the vector $$R(h^{-1}g)=R(h_m^{-1}h_{m-1}^{-1}\cdots h_{p+2}^{-1}cg_{p+2}\cdots g_{n-1}g_n)$$ are all in the different orthogonal subspaces $\mathcal{H}_t$'s of $\mathcal{H}$. Therefore we have,
\begin{equation*}
    \begin{array}{rl}
         k(h^{-1}g,e)
         =& \sum_{i=p+2}^n||R_i(g_i)||^2+ \sum_{i=p+2}^m||R_i(h^{-1}_i)||^2 +||R_{p+1}(h_{p+1}^{-1}g_{p+1})||^2\\
         \\
         =& \sum_{i=p+2}^n||R_i(g_i)||^2+ \sum_{i=p+2}^m||R_i(h_i)||^2  
        +||R_{p+1}(g_{p+1})-R_{p+1}(h_{p+1})||^2\\
        &\qquad\qquad \qquad\qquad\qquad\qquad\qquad\qquad\qquad[\text{as}~\phi_A,\phi_B~\text{are symmetric}]\\
        =& ||\oplus_{i=p+1}^n R_i(g_i)-\oplus_{i=p+1}^m R_i(h_i)||^2\\
        \\
        =& ||R(g)-R(h)||^2=k(g,h)
    \end{array}
\end{equation*}
This proves the $G$-invariance of the kernel $k$.
\end{proof}

Now, the kernel $k$ induces a conditionally negative definite function $\tilde{\phi}$ on $G$, which may not be proper. To make it proper, we add the reduced length function $l_r$, which was proven to be CND by Haagerup, in \cite{haagerup1978example}. Define a function $\phi: G\rightarrow\mathbb{R}$ by 
$$g\mapsto \phi(g):=l_r(g)+\tilde{\phi}(g)$$
see lemma \ref{properness}, for the proof of properness of $\phi$. Therefore $(G,\phi,\mathcal{H},R)$ is a Haagerup tuple.

\subsection*{Main Result}\label{main result}
Given a finite graph $\Gamma$, we have a family $\{G_v\}_{v\in V(\Gamma)}$ of groups satisfying Haagerup property and associated Haagerup tuples $\{(G_v,\phi_v,\mathcal{H}_v,R_v)\}_{v\in V(\Gamma)}$. We follow similar approach as above to prove theorem \ref{main theorem} for the graph product $G(\Gamma)$.\\

Let $T$ to be the set of all cosets $\{G(\Gamma)/G(st(v))\}_{v\in V(\Gamma)}$. Consider the Hilbert space $\mathcal{H}:=\oplus_{t\in T} \mathcal{H}_t$, where $\mathcal{H}_t$ is $\mathcal{H}_v$ if $t$ is a coset $gG(st(v))$. Note that the Hilbert space $\mathcal{H}$ in the free product case pertains to the disconnected graph with two vertices. Again without loss of generality, we can assume that each $R_v$ maps identity element to the zero vector in the corresponding Hilbert spaces.\\

A non-trivial $g\in G(\Gamma)$ has a reduced form  $g=g_1g_2\ldots g_m$, where $g_i\in G_{v_i}$ for $i=1,2,\dots,m$. For simplicity we write $g_i\in   G(i)\leq G(st(i))$. Similar to the definition (\ref{definition of R}) in the free product case, we define $R:G(\Gamma)\rightarrow \mathcal{H}$ by:
$$
R(g):=
\begin{cases}
\oplus_{i=1}^m R_i(g_i)_{g_1g_2\ldots g_{i-1}G(st(i))}, ~~\text{if}~g\neq e\\
0_{\mathcal{H}}~~~~~~~~~~~~~~~~~~~~~~~~~~~~~~~~~~\text{if}~g=e\\
\end{cases}
$$
where $R_i$ is $R_{v_i}$. To verify that $R$ is well defined: suppose the groups $G(i)$ and $G(i+1)$ commute in $G(\Gamma)$ and we can perform a shuffle at the $i$-th position of reduced form of $g$, i.e.
$$g=g_1g_2\ldots g_{i-1}g_ig_{i+1}\ldots g_m=g_1g_2\ldots g_{i-1}g_{i+1}g_{i}\ldots g_m.$$ Therefore one can say 
\begin{equation*}
    \begin{array}{rl}
         g_1g_2\ldots g_{i-1} g_i G(st(i+1)) =&  g_1g_2 \ldots g_{i-1} G(st(i+1)) ~~\text{and}  \\
         g_1g_2\ldots g_{i-1}g_{i+1} G(st(i)) = &  g_1g_2\ldots g_{i-1} G(st(i))
    \end{array}
\end{equation*}
showing that a shuffle between $g_i$ and $g_{i+1}$ results in a permutation of the $i$-th and $(i+1)$-th component of the vector $R(g)$ in the direct sum, however leaves $R(g)$ unaffected. A reduced form of $g$ can be obtained from another by only a finitely many shuffles. So the map $R$ is well defined on $G(\Gamma)$. 
\begin{lemma}
    The kernel $k$ on $G(\Gamma)\times G(\Gamma)$ given as follows is $G(\Gamma)$-invariant:  $$(g,h)\mapsto ||R(g)-R(h)||^2$$
\end{lemma}
\begin{proof}
    Note that the reduced forms are not unique as in the free product case. Otherwise the proof is exactly the same as lemma \ref{group inv kernel in free product}.   
\end{proof}

Thus we obtain a conditionally negative definite function  $\tilde{\phi}$ on $G(\Gamma)$, which is not necessarily proper. Notice that if $g$ has a reduced form $g=g_1g_2\dots g_m$, where $g_i\in G_{v_i}$, then 
\begin{equation}\label{bozejko type splitting}
    \tilde{\phi}(g)=\sum_{i=1}^m \phi_{v_i}(g_i)
\end{equation}

From equation (\ref{comparison metric}), we observe that the distance function $d_r$ is conditionally negative definite on $G(\Gamma)\times G(\Gamma)$, hence $l_r$ is also CND on $G(\Gamma)$.\\

Define $\phi_\Gamma$ on $G(\Gamma)$ as:
$$\phi_\Gamma:=l_r+\tilde{\phi}$$ being a sum of two CND functions, it is also conditionally negative definite on $G(\Gamma)$.

\begin{remark}
    For any $v\in V(\Gamma)$ one has, $\phi_\Gamma\Big|_{G_v}\equiv 
     1+ \phi_v$.    
\end{remark}

\begin{lemma}\label{properness}
    $\phi_\Gamma$ is proper on $G(\Gamma)$.
\end{lemma}
\begin{proof}
    Let $|g|$ denotes the word length of $g \in G(\Gamma)$. We need to prove that $\phi_\Gamma(g)\rightarrow \infty$ as $|g|$ grows to infinity. We have two cases to prove.\\
    
    \textbf{Case 1:} $|g|$ goes to infinity, but $l_r(g)$ is finite, say $l_r(g)=k$. Let $g=g_1g_2\cdots g_k$ is a reduced form. Then by pigeon-hole principle there is an $1\leq i\leq k$, such that $|g_i|$ grows to infinity. Then the value of $\phi_{v_i}(g_i)$ also goes to infinity, as $\phi_{v_i}$ is proper. Hence by equation (\ref{bozejko type splitting}), $\phi_\Gamma(g)$ goes to infinity.\\
    
    \textbf{Case 2:}  $|g|$ goes to infinity and $l_r(g)$ also goes to infinity. So again, $\phi_\Gamma(g)=l_r(g)+\tilde{\phi}(g)$ goes to infinity. 
\end{proof}
Hence we have the following result (Theorem \ref{main theorem}):
\begin{theorem}
    Suppose $\Gamma$ is a finite graph and $\{G_v\}_{v\in V(\Gamma)}$ is a collection of groups with the Haagerup property. Then the graph product $G(\Gamma)$ has Haagerup property.
\end{theorem}

\section{Acknowledgement}
PSG\footnote{parthasarathi.ghosh.100@gmail.com} was supported by CSIR Fellowship [File No. 08/155(0066)/2019-EMR-I], Govt. of India. SD\footnote{shubhabrata.maths@presiuniv.ac.in} acknowledges the infrastructural support provided by the Dept. of Mathematics, Presidency University, through the DST-FIST [File No. SR/FST/MS-I/2019/41].

\bibliographystyle{alpha}
\bibliography{references}

\end{document}